\documentclass[a4paper]{article}

\usepackage{amsfonts}
\usepackage{amsmath}
\usepackage{mathrsfs}
\usepackage{amssymb,color}

\usepackage[latin1]{inputenc}
\usepackage{amsthm}
\usepackage{graphicx}
\usepackage{enumerate}
\newtheorem{theorem}{Theorem}[section]
\newtheorem{proposition}[theorem]{Proposition}
\newtheorem{definition}[theorem]{Definition}
\newtheorem{lemma}[theorem]{Lemma}

\begin{document}
	\title{\textbf{ Disconjugacy characterization by means of spectral of $(k,n-k)$ problems.}}
	\date{}
	\author{Alberto Cabada\footnote{Partially supported by Ministerio de Educaci\'on y Ciencia, Spain and FEDER, projects MTM2010-15314 and MTM2013-43014-P.}   \;and  Lorena Saavedra\\Department of Mathematical Analysis,\\ Faculty of Mathematics,\\University of Santiago de Compostela,\\ Santiago de Compostela, Galicia,
		Spain\\
		alberto.cabada@usc.es, lorena.saavedra@usc.es
	} 
	\maketitle
	
	\begin{abstract}
			This paper is devoted to the description of the interval of parameters for which  the general linear $n^{\rm th}$-order equation 
			\begin{equation}
			\label{e-Ln}
			T_n[M]\,u(t) \equiv u^{(n)}(t)+a_1(t)\, u^{(n-1)}(t)+\cdots +a_{n-1}(t)\, u'(t)+(a_{n}(t)+M)\,u(t)=0 \,,\quad t\in I\equiv[a,b],
			\end{equation}
			with $a_i\in C^{n-i}(I)$, is disconjugate on	$ I $.			
			Such interval is characterized by the closed  to zero eigenvalues of this problem coupled with $(k,n-k)$ boundary conditions, given by
			\begin{equation}
			\label{e-k-n-k}
			u(a)=\cdots=u^{(k-1)}(a)=u(b)=\cdots=u^{(n-k-1)}(b)=0\,,\quad 1\leq k\leq n-1\,.
			\end{equation}
			
	\end{abstract}

\section{Introduction}

There are a huge number of works related to the disconjugacy and its properties, (see \cite{Cop, Lev, Sha} for details). From the last half of the past century till now this subject of investigation has attracted important researchers who have stablished very interesting criteria to ensure such property for particular equations. 

This is the case of  \cite{Neh2}, where is characterized the disconjugacy of the second order equation $y''(t)+f(t)\, y(t)=0$, with $f \ge 0$, in terms of the least eigenvalue of the corresponding eigenvalue problem
$y''(t)+\lambda \, f(t)\, y(t)=0, y(a)=y'(b)=0$. In \cite{Barr}  the general equation $(p(t) y'(t))'+f(t)\, y(t)=0$ has been treated. On that paper some characterization by means of variational approach has given. More recently, in \cite{ClaHin} sufficient conditions to ensure the disconjugation of some second and fourth order equations are showed.

In \cite{Neh1} sufficient conditions (and different  necessary ones) for disconjugacy on $[a,+\infty)$, are obtained for linear differential equations of the form $y^{(2\,n)}(t)-(-1)^n\,p(t)\,y(t)=0$, with $p\ge 0$. 
Also, in \cite{Eli} conditions for disconjugacy of the linear differential equation $y^{(n)}(t)+p(t)\,y(t)=0 $ in $[a,+\infty)$ and $p(t)$ of constant sign, are given. In \cite{Sak} the disconjugacy of the linear differential equation $(r(t)\,x'(t))'+p(t)\,x'(t)+q(t)\,x(t)=0$, in $[a,b]$ is studied. Sufficient conditions to warrant the disconjugacy of the nonlinear 
 $p-$ Laplacian equation $\left( |u'(t)|^{p-1}\,u'(t)\right) '+q(t)\,|u(t)|^{p-2}\,u(t)=0$ on $[a,+\infty)$, have been obtained in\cite{NapPin}.

The aim of this paper consists on the  characterization of the disconjugacy of the general $n$-th order linear differential operator $u^{(n)}(t)+a_1(t)\, u^{(n-1)}(t)+\cdots +a_{n-1}(t)\, u'(t)+a_{n}(t)\,u(t)$ on any arbitrary interval $[a,b]$.

Tacking into account that the coefficient of $u$ can be uniquely decomposed as 
$$a_n(t)=\tilde{a}_n(t)+\frac{1}{b-a}\int_a^ba_n(s)\,ds, \quad t \in I,$$
it is obvious that such problem is equivalent to  study the set of parameters $M$ for which the linear differential equation \eqref{e-Ln} is disconjugate on $I$. To this end, we assume that such set is not empty, i.e.,  there exists at least a $\bar{M}$ such that $T_n[\bar{M}]\,u(t)=0$ is disconjugate on $I$.

Denote as $X_k$ the corresponding space of definition to the  $(k,n-k)$ boundary value conditions given in \eqref{e-k-n-k}.

{ \begin{equation*}
	X_k= \left\lbrace u\in C^n(I)\ \mid u(a) =\dots=u^{(k-1)}(a)=u(b)=\dots=u^{(n-k-1)}(b)=0\right\rbrace.
	\end{equation*}}

Now, in order to make the paper more readable, we introduce some previous concepts and results.

\begin{definition}
	Let $a_k\in C^{n-k}(I)$ for $k=1,\dots,n$. The linear differential equation (\ref{e-Ln}) of order $n$ is said to be disconjugate on an interval $J$ if every non trivial solution has, at most, $n-1$ zeros on $J$, multiple zeros being counted according to their multiplicity.
\end{definition}
	
%
%	\begin{definition}
%		The functions $u_1,\dots, u_n \in C^n(I)$ are said to form a Markov system on the interval $I$ if the $n$ Wronskians
%		\begin{equation}
%		W(u_1,\dots,u_k)=\left| \begin{array}{ccc}
%		u_1&\cdots&  u_k\\
%		\vdots&\cdots&\vdots\\
%		u_1^{(k-1)}&\cdots&u_k^{(k-1)}\end{array}\right| \,,\quad k=1,\dots,n \,,
%		\end{equation}
%		are positive throughout $I$.
%	\end{definition}
	The following results and definitions about this concept are collected on \cite[Chapter 3]{Cop}.
	
%	\begin{theorem}\label{T::4}
%		The linear differential equation (\ref{e-Ln}) has a Markov fundamental system of solutions on the compact interval $I$ if, and only if, it is disconjugate on $I$.
%		
%	\end{theorem}
	
	\begin{theorem}\label{T::1}
		If the equations $L_1\,y=0$ and $L_2\,y=0$ are disconjugate on the interval $I$, then the composite equation $L_1\,(L_2\,y)=0$ is also disconjugate on $I$.
	\end{theorem}
	
	Defining the distance between two equations \eqref{e-Ln}$_1$ and \eqref{e-Ln}$_2$ by
	$\sup_{t\in I}\sum_{k=1}^n \left| a_{k,1}(t)-a_{k,2}(t)\right|$,
	we have the following result in the correspondent metric space. 
	\begin{proposition}
		\label{P::1}
		The set of all disconjugate equations (\ref{e-Ln}) on a compact interval $I$ is connected and open.
	\end{proposition}

		\begin{definition}\label{Def::1}
			Let $a\in\mathbb{R}$, denote the first right point conjugate  of $a$ for the linear differential equation (\ref{e-Ln})  by
			\[\eta_M(a)=\sup\left\lbrace b>a\quad\mid\quad \text{equation (\ref{e-Ln}) is disconjugate on } [a,b]\right\rbrace \in (a,\infty]\,.\]
		\end{definition}
		
		We consider a fundamental system of solutions $y_1[M](t),\dots,y_n[M](t)$ of equation \eqref{e-Ln}, where every $y_k[M](t)$ is uniquely determined by the following initial conditions:
		
		\[y_k^{(n-k)}[M](a)=1\,,\quad y_k^{(n-j)}[M](a)=0\,,\ j=1,\dots,n\,,\ j\neq k\,.\]
		
		Then, we denote  the $n-1$ Wronskians as

			\begin{equation}
					W^n_k[M](t):=\left| \begin{array}{ccc}
						y_1[M](t)&\dots&y_k[M](t)\\
					\vdots&\cdots&\vdots\\
					y_1[M]^{(k-1)}(t)&\cdots&y_k[M]^{(k-1)}(t)\end{array}\right| \,,\quad k=1,\dots,n-1 \,.
					\end{equation}
					
			\begin{proposition}\label{P::3}
			There exists a solution of equation (\ref{e-Ln}) which verifies the boundary conditions $(k,n-k)$ on $[a,b]$ if, and only if, $W^n_{n-k}[M](b)=0$.
			\end{proposition}
		\begin{definition}
		Denote $\omega_M(a)$ as the least $b>a$, if one exists, at which one of the Wronskians $W^n_1[M](b),\dots,$ $W^n_{n-1}[M](b)$ vanishes.
		\end{definition}
		
		The next result gives us a relation between this concept and the one given on Definition \ref{Def::1}.
		
	\begin{proposition}\label{P::4}
		$\eta_M(a)=\omega_M(a)$.
	\end{proposition}
	
	\begin{proposition}\label{P::2}
		Let $b=\eta_M(a)$ and let $n-k\in\left\lbrace 1,\dots,n-1\right\rbrace $ be   such that $W^n_{n-k}[M](b)=0$ and $W^n_{\ell}[M](b)\neq 0$ for every $\ell<n-k$. The corresponding solution of \eqref{e-Ln} with $(k,n-k)$ boundary conditions is uniquely determined up to a constant factor, and does not vanish on the open interval $(a,b)$.
	\end{proposition}
	
	Now, we are going to introduce the concept of Green's function related to the operator $T_n[M]$ coupled with boundary conditions (\ref{e-k-n-k}), see \cite{Cab} for details.
	
	\begin{definition}
		We say that $g_M$ is a Green's function for problem (\ref{e-Ln})-(\ref{e-k-n-k}) if it satisfies the following properties:
		\begin{itemize}
			\item[$(g_1)$] $g_M$ is defined on the square $I\times I$ (except $t=s$ if $n=1$).
			\item[$(g_2)$] For $k=0,1,\dots,n-2$, the partial derivatives $\dfrac{\partial^k g_M}{\partial t^k}$ exist and they are continuous on $I\times I$.
			\item[$(g_3)$] $\dfrac{\partial^{n-1} g_M}{\partial t^{n-1}}$ and $\dfrac{\partial^n g_M}{\partial t^n}$ exist and they are continuous on the triangles $a\leq s<t\leq b$ and $a\leq t<s\leq b$.
			\item[$(g_4)$] For each $s\in(a,b)$, the function $t\rightarrow g_M(t,s)$ is a solution of the differential equation (\ref{e-Ln})  on $[a,s)\cup (s,b]$.
			\item[$(g_5)$] For each $t\in(a,b)$ there exist the lateral limits
		{\footnotesize	\[\dfrac{\partial^{n-1} }{\partial t^{n-1}}g_M(t^-,t)=\dfrac{\partial^{n-1} }{\partial t^{n-1}}g_M(t,t^+)\quad\text{and}\quad \dfrac{\partial^{n-1} }{\partial t^{n-1}}g_M(t,t^-)=\dfrac{\partial^{n-1} }{\partial t^{n-1}}g_M(t^+,t)\]}
		and, moreover
	{\small 	\[\dfrac{\partial^{n-1} }{\partial t^{n-1}}g_M(t^+,t)-\dfrac{\partial^{n-1} }{\partial t^{n-1}}g_M(t^-,t)=\dfrac{\partial^{n-1} }{\partial t^{n-1}}g_M(t,t^-)-\dfrac{\partial^{n-1} }{\partial t^{n-1}}g_M(t,t^+)=1\,.\]}
	\item[$(g_6)$] For each $s\in(a,b)$, the function $t\rightarrow g_M(t,s)$ satisfies the boundary conditions $(k,n-k)$, i.e.,
{\small 	\[g_M(a,s)=\cdots=\dfrac{\partial^{k-1} }{\partial t^{k-1}}g_M(a,s)=g_M(b,s)=\cdots=\dfrac{\partial^{n-k-1} }{\partial t^{n-k-1}}g_M(b,s)\,.\]}
		\end{itemize}
	\end{definition}
	
		Denote the Green's function related to  the operator $T_n[M]$ in $X_k$ as $g_{M,k}$.
	
	If equation \eqref{e-Ln} is disconjugate on $I$ and $u$ is a solution of problem $T_n[M]\,u(t)=\sigma(t)$, $t\in I$, with boundary conditions $(k,n-k)$, it is uniquely determined by the expression 
	\[u(t)=\int_a^bg_{M,k}(t,s)\,\sigma(s)\,ds\,.\]

	We  also mention a  result which appears on \cite[Chapter 3, Section 6]{Cop} and that connects the disconjugacy and the sign of Green's function related to  problem (\ref{e-Ln})-(\ref{e-k-n-k}).
	
	\begin{lemma}\label{L::1}
		If the linear differential equation (\ref{e-Ln}) is disconjugate on $I$ and $g(t,s)$ is the Green's function related to problem (\ref{e-Ln})-(\ref{e-k-n-k}), by defining $p(t)=(t-a)^k\,(t-b)^{n-k}$ we have that
	\[
		g(t,s)\,p(t)\geq 0\,,\quad \forall\,(t,s)\in I \times I\quad \text{and}\quad
		\dfrac{g(t,s)}{p(t)}>0\,,\quad \forall\,(t,s)\in I\times (a,b)\,.\]
	\end{lemma}
	In the sequel, we introduce two conditions on $g_M(t,s)$ that will be used along the paper, see \cite[Section 1.8]{Cab}.
	\begin{itemize}
		\item[$(P_g$)] Suppose that there is a continuous function $\phi(t)>0$ for all $t\in (a,b)$ and $k_1,\ k_2\in \mathcal{L}^1(I)$, such that $0<k_1(s)<k_2(s)$ for a.e. $s\in I$, satisfying
		\[\phi(t)\,k_1(s)\leq g_M(t,s)\leq \phi(t)\, k_2(s)\,,\quad \text{for a. e. } (t,s)\in I \times I \,.\]
		\item[($N_g$)] Suppose that there is a continuous function $\phi(t)>0$ for all $t\in (a,b)$ and $k_1,\ k_2\in \mathcal{L}^1(I)$, such that $k_1(s)<k_2(s)<0$ for a.e. $s\in I$, satisfying
		\[\phi(t)\,k_1(s)\leq g_M(t,s)\leq \phi(t)\, k_2(s)\,,\quad \text{for a. e. }(t,s)\in I \times I\,.\]
	\end{itemize}
	Next result, which appears in \cite{CabSaa}, gives us a property of the operator under the disconjugacy hypothesis.
	
	\begin{lemma}\label{L::2}
		Let $\bar{M}\in\mathbb{R}$ be such that $T_n[\bar{M}]\,u(t)=0$ is disconjugate on $I$. Then the following properties are fulfilled:
		\begin{itemize}
			\item If $n-k$ is even, then $T_n[\bar{M}]$ is a inverse positive operator on $X_k$ and its related Green's function, $g_{\bar{M}}(t,s)$, satisfies ($P_g$).
			
			\item If $n-k$ is odd, then $T_n[\bar{M}]$ is a inverse negative operator on $X_k$  and its related Green's function satisfies ($N_g$).
			
		\end{itemize}
	\end{lemma}

		The following result, which appears on \cite[Theorem 3.2]{Neh}, shows a property of the eigenvalues of a disconjugate operator.
		
			\begin{theorem}\label{T::10}
				Let $\bar{M}\in\mathbb{R}$ be such that $T_n[\bar{M}]\,u(t)=0$ is disconjugate on $I$. Then
				\begin{itemize}
					\item If $n-k$ is even, there is not any eigenvalue of $T_n[\bar{M}]$ on $X_k$  such that $\lambda<0$.
					\item If $n-k$ is odd, there is not any eigenvalue of $T_n[\bar{M}]$ on $X_k$  such that $\lambda>0$.
				\end{itemize} 
			\end{theorem}

		Next two following results, see \cite[Section 1.8]{Cab}, ensure the existence of eigenvalues in different cases
		
			\begin{theorem}
				Let $\bar{M}\in \mathbb{R}$ be fixed. If operator $T_n[\bar{M}]$ is invertible in $X_k$ and its related Green's function satisfies condition $(P_g)$, then there exists $\lambda_1>0$, the least eigenvalue in absolute value of operator $T_n[\bar{M}]$ in $X_k$. Moreover, there exist a nontrivial constant sign eigenfunction corresponding to the eigenvalue $\lambda_1$.
			\end{theorem}	
			
		\begin{theorem}
			Let $\bar{M}\in \mathbb{R}$ be fixed. If operator $T_n[\bar{M}]$ is invertible in $X_k$ and its related Green's function satisfies condition $(N_g)$, then there exists $\lambda_2<0$, the least eigenvalue in absolute value of operator $T_n[\bar{M}]$ in $X_k$. Moreover, there exist a nontrivial constant sign eigenfunction corresponding to the eigenvalue $\lambda_2$.
		\end{theorem}	
			
	Finally, we introduce the following sets, which characterize the intervals of constant sign for $g_M(t,s)$.
	\begin{eqnarray}
	\nonumber	P_T&=& \left\lbrace M\in \mathbb{R}\,,\ \mid \quad g_M(t,s)\geq 0\quad \forall \quad(t,s)\in I\times I\right\rbrace, \\\nonumber
		N_T&=& \left\lbrace M\in \mathbb{R}\,,\ \mid \quad g_M(t,s)\leq 0\quad \forall\quad  (t,s)\in I\times I\right\rbrace.
	\end{eqnarray}
	
	Next results describe the structure of the two previous parameter's set, see \cite[Section 1.8]{Cab}
	\begin{theorem}\label{T::6}
		Let $\bar{M}\in\mathbb{R}$ be fixed. Suppose that operator $T_n[\bar{M}]$ is invertible on $X_k$, its related Green's function is nonnegative on $I \times I$, it satisfies condition ($P_g$), and the set $P_T$ is bounded from above.
		Then $P_T=(\bar{M}-\lambda_1,\bar{M}-\bar{\mu}]$, with $\lambda_1>0$ the least positive eigenvalue of operator $T_n[\bar{M}]$ in $X_k$ and $\bar{\mu}<0$ such that $T_n[\bar{M}-\bar{\mu}]$ is invertible in $X_k$ and the related nonnegative Green's function $g_{\bar{M}-\bar{\mu}}$ vanishes at some points on the square $I\times I$.
	\end{theorem}
	
	\begin{theorem}\label{T::7}
		Let $\bar{M}\in\mathbb{R}$ be fixed. Suppose that operator $T_n[\bar{M}]$ is invertible in $X_k$, its related Green's function is nonpositive on $I\times I$, it satisfies condition ($N_g$), and the set $N_T$ is bounded from below.
		Then $N_T=[\bar{M}-\bar{\mu},\bar{M}-\lambda_2)$, with $\lambda_2<0$ the biggest negative eigenvalue of operator $T_n[\bar{M}]$ in $X_k$ and $\bar{\mu}>0$ such that $T_n[\bar{M}-\bar{\mu}]$ is invertible in $X_k$ and the related nonpositive Green's function $g_{\bar{M}-\bar{\mu}}$ vanishes at some points on the square $I\times I$.
	\end{theorem}

	\section{Characterization of disconjugacy}
	
	This section is devoted to prove de main result of this paper. The result is the following.
	\begin{theorem}\label{T::aut}
		Let $\bar{M}\in\mathbb{R}$ and $n\geq 2$ be such that $T_n[\bar{M}]\,u(t)=0$ is a disconjugate equation on $I$. Then, $T_n[M]\,u(t)=0$ is a disconjugate equation on $I$ if, and only if, $M\in(\bar{M}-\lambda_1,\bar{M}-\lambda_2)$, where
		\begin{itemize}
			\item  $\lambda_1=+\infty$ if $n=2$ and, for $n>2$, $\lambda_1>0$ is the minimum of the least positive eigenvalues on $T_n[\bar{M}]$ in $X_k$, with $n-k$  even.
			\item $\lambda_2<0$ is the maximum of the biggest negative eigenvalues on $T_n[\bar{M}]$ in $X_k$, with $n-k$  odd.
		\end{itemize}
	\end{theorem}
	\begin{proof}

		Let $n>2$. First, we are going to see that the optimal interval of disconjugation, $D_{\bar{M}}$, must necessarily be a subset of $(\bar{M}-\lambda_1,\bar{M}-\lambda_2)$.
		
		Using, Lemma \ref{L::1}, it is known that if $M\in D_{\bar{M}}$, Green's function related to operator $T_n[M]$ in $X_k$ is of constant sign, positive if $n-k$ is even and negative if $n-k$ is odd. 
		
		Let $\widehat{k}\in\left\lbrace 1,\dots,n-1\right\rbrace $ be such that $n-\widehat{k}$ is even and $\lambda_1$ attained as the least positive eigenvalue of $T_n[\bar{M}]$ in $X_{\widehat{k}}$. Using Lemma \ref{L::2} and Theorem \ref{T::6}, we can affirm that $g_{\widehat{M},\widehat{k}}$ changes sign on $I\times I$ for $\widehat{M}\leq\bar{M}-\lambda_1$, then $\widehat{M}\notin D_{\bar{M}}$ for every $\widehat{M}\leq \bar{M}-\lambda_1$.
		
		In an analogous way, let $\widetilde{k}\in\left\lbrace  1,\dots,n-1\right\rbrace $ be such that $n-\widetilde{k}$ is odd and $\lambda_2$ attained as the biggest negative eigenvalue of $T_n[\bar{M}]$ in $X_{\widetilde{k}}$. Using the same arguments, with Lemma \ref{L::2} and Theorem \ref{T::7}, we can affirm that $g_{\widetilde{M},\widetilde{k}}$ has not constant sign on $I\times I$ for $\widetilde{M}\geq\bar{M}-\lambda_2$, then $\widetilde{M}\notin D_{\bar{M}}$ for every $\widetilde{M}\geq \bar{M}-\lambda_2$.
		
		Hence, we have proved that $D_{\bar{M}}\subset\left( \bar{M}-\lambda_1,\bar{M}-\lambda_2\right) $.
		
		\vspace{0.3cm}
		
		Let's see now that $D_{\bar{M}}=\left( \bar{M}-\lambda_1,\bar{M}-\lambda_2\right) $. Denote $M_1=\inf D_{\bar{M}}$ and $M_2=\sup D_{\bar{M}}$. Because of Proposition \ref{P::1}, $D_{\bar{M}}$ should be an open interval, in particular $M_j\neq\bar{M}$ for $j=1,2$.
		
	If $D_{\bar{M}}\neq (\bar{M}-\lambda_1,\bar{M}-\lambda_2)$ then, at least one (or both) of the two following inequalities holds: either $M_1>\bar{M}-\lambda_1$ or $M_2<\bar{M}-\lambda_2$.
	
		Suppose that first inequality is fulfilled.
		
		Since $T_n[M_1]\,u(t)=0$ is not a disconjugate equation on $I$, we have that  $c=\eta(a)\leq b$.
		
		Using Proposition \ref{P::2}, we can ensure the existence of $\ell\in\left\lbrace 1,\dots,n-1\right\rbrace $ such that there exists a solution of $T_n[M_1]\,u(t)=0$, satisfying boundary conditions $(n-\ell,\ell)$ on $[a,c]$.
		
		If $c=b$, we have that $\bar{M}-M_1\in(\lambda_2,\lambda_1)$ will be an eigenvalue of $T_n[\bar{M}]$ in $X_\ell$, and it contradicts the definition of $\lambda_1$ when $n-l$ is even and $\lambda_2$, if  $n-l$ is odd.
		
		So, we have that $c<b$.
		
		Using Proposition \ref{P::3}, we know that $W_\ell[M_1](c)=0$. 
		
		And, since $T_n[M]\,u(t)=0$ is a disconjugate equation on $I$ for $M\in (M_1,M_2)$, we can affirm that $W_\ell[M_1+\delta](t)\neq0\,,\quad t\in(a,b]$ for every $0<\delta<M_2-M_1$.
		
		Since $W_\ell[M](t)$ is a continuous function of $M$, we can affirm that $W_\ell[M_1](t)$ is of constant sign on a neighborhood of $c$, so it has a double zero at $c$ as a function of $t$.

			Using the expression of the derivative of the Wronskian given in \cite{Neh}, we know that
			
		\begin{equation}\label{Ec::dw}0=\dfrac{\partial}{\partial t}W^n_{\ell}[M_1](t)_{\mid t=c}=\left| \begin{array}{ccc}
				y_1[M_1](c)&\dots &y_\ell[M_1](c)\\
					&\vdots&\\
				y_1^{(\ell-2)}[M_1](c)&\dots&y_\ell^{(\ell-2)} [M_1](c)\\
			
				y_1^{(\ell)}[M_1](c)&\dots&y_\ell^{(\ell)}[M_1](c)\end{array}\right|.\end{equation}

We take the following solution of \eqref{e-Ln}
\[y(t)=\left| \begin{array}{ccc}
y_1[M_1](c)&\dots &y_\ell[M_1](c)\\
&\vdots&\\
y_1^{(\ell-2)}[M_1](c)&\dots&y_\ell^{(\ell-2)} [M_1](c)\\

y_1[M_1](t)&\dots&y_\ell[M_1](t)\end{array}\right|\,.\]

Since it is a linear combination of $y_1[M_1],\dots y_\ell[M_1]$, it is obvious that it has $n-\ell$ zeros at $a$.

Since $W_\ell[M_1](c)=0$, $y$ trivially verifies the boundary conditions $(n-\ell,\ell)$ at $[a,c]$. And, using Proposition \ref{P::2}, since $c=\eta_M(a)$ we know that it does not vanish on the open interval $(a,c)$.

 Because of equality (\ref{Ec::dw}) it is not difficult to verify that such function also verifies the boundary conditions $(n-\ell-1,\ell+1)$ on $[a,c]$.

%Let's see what happens on $c$. Since $W_\ell^n[M_1](c)=0$, it is verified that $y^{(\ell-1)}(c)=0$. 
%
%Now, we are going to solve the following linear system
%
%\begin{eqnarray}
%\nonumber x_1\,y_1[M_1](c)+\cdots+x_{\ell-1}\,y_{\ell-1}(c)&=&-\alpha_\ell\,y_\ell[M_1](c)\,,\\
%\nonumber&\vdots&\\
%\nonumber x_1\,y_1^{(\ell-2)}[M_1](c)+\cdots+x_{\ell-1}\,y_{\ell-1}^{(\ell-2)}(c)&=&-\alpha_\ell\,y_\ell^{(\ell-2)}[M_1](c)\,.
%\end{eqnarray}
%
%Since $W_{\ell-1}^n[M_1](c)=\alpha_\ell\neq 0$, we can use the Cramer's rule for every $k=1,\dots,\ell-1$, to obtain
%
%{\scriptsize\[ x_k=\dfrac{\left| \begin{array}{cccccccc}
% y_1[M_1](c)&\cdots&y_{k-1}[M_1](c)&-\alpha_\ell\,y_{\ell}[M_1](c)&y_{k+1}[M_1](c)&\cdots& y_{\ell-1}[M_1](c)\\
%\vdots&\cdots&&\vdots&&\cdots&\vdots\\
%y_1[M_1](c)&\cdots&y_{k-1}[M_1](c)&-\alpha_\ell\,y_{\ell}[M_1](c)&y_{k+1}[M_1](c)&\cdots& y_{\ell-1}[M_1](c) \end{array}\right|}{\alpha_\ell} =\alpha_k\,\]}
%
%So, $y$ verifies the boundary conditions $(n-\ell,\ell)$.
%
%Because of equality (\ref{Ec::dw}), we can affirm that $y^{(\ell)}(c)=0$. Hence, $y(t)$ also verifies boundary conditions $(n-\ell-1,\ \ell+1)$, and using Proposition \ref{P::9} we know that it is of constant sign. 

As consequence, denoting as $g_{\bar{M},n-\ell}$ and $g_{\bar{M},n-\ell-1}$,  the related Green's functions to problem \eqref{e-Ln} -- \eqref{e-k-n-k}, for $M=\bar{M}$, $b=c$ and $k=l$ or $k=l+1$, respectively, we deduce the following equalities for all $t\in [a,c]$

\[
y(t)=\int_a^c g_{\bar{M},n-\ell}(t,s)\,(\bar{M}-M_1)\,y(s)\,ds\,,\quad \text{and} \quad
y(t)=\int_a^c g_{\bar{M},n-\ell-1}(t,s)\,(\bar{M}-M_1)\,y(s)\,ds\,.\]
Using Lemma \ref{L::1} we know that $g_{\bar{M},n-\ell}(t,s)$ and $g_{\bar{M},n-\ell-1}(t,s)$  have different constant sign on $[a,c]\times(a,c)$, so last equalities cannot be satisfied at the same time. Then we can affirm that $M_1=\bar{M}-\lambda_1$.

With analogous arguments we conclude that $M_2=\bar{M}-\lambda_2$. 

If $n=2$, the argument related to $\lambda_2$ is the same. 

Suppose that there exist $M^*<\bar{M}$ such that the equation \eqref{e-Ln} is not disconjugate on $[a,b]$, then $M_1<\bar{M}$ must be defined and also $c=\eta_M(a)$. If $c=b$ it implies the existence of a positive eigenvalue of $T_n[\bar{M}]$ in $X_1$, which is a contradicts Theorem \ref{T::10}. 

Then, we can proceed analogously to the case where $n>2$ with $c<b$ and arrive to a contradiction. So, our result is proved.
\end{proof}

\subsection{Particular cases}
Since $u^{(n)}(t)=0$ is always a disconjugate equation at any interval (see \cite{Cop} for details), this result can obviously be applied to  operators $T_n[M]\,u(t)=u^{(n)}(t)+M\,u(t)$. So, in order to  construct the optimal parameter set of disconjugacy, we only need to calculate the closest to zero eigenvalues.

Until eighth order the eigenvalues of problems $(k,n-k)$ are explicitly obtained on \cite[Section 4]{CabSaa}.
For instance, in the second order case we know that the closest to zero eigenvalue of $u''$ in $X_1$ is $-\pi^2$, so the optimal interval of disconjugacy is $(-\infty,\pi^2)$.

Also, in third order we have  that  the least positive eigenvalue of operator $u'''$ in $X_1$  is $\left( \lambda_3^1\right)  ^3$ and the biggest negative eigenvalue of operator $u'''$ in $X_2$  is $-\left(  \lambda_3^1\right)  ^3$, where $ \lambda_3^1\approxeq 4.23321$ is the least positive solution of
\begin{equation*}
\cos \left(\frac{1}{2} \sqrt{3} \lambda \right)-\sqrt{3} \sin \left(\frac{1}{2} \sqrt{3} \lambda  \right)=e^{\frac{-3\,\lambda}{2}}\,.
\end{equation*}
So, we can affirm that $u'''(t)+M\,u(t)=0$ is a disconjugate equation if, and only if $M\in \left( -\left(  \lambda_3^1\right)  ^3,\left(  \lambda_3^1\right)  ^3\right) $.

In fourth order we obtain that  the biggest negative eigenvalue of operator $u^{(4)}$ in $X_1$ and $X_3$ is given by $-\left( \lambda_4^1\right) ^4$ and the least positive eigenvalue of operator $u^{(4)}$ in $X_2$ is $\left( \lambda_4^2\right) ^4$, where $\lambda_4^1\approxeq 5.553$ is the least positive solution of
\begin{equation*}
\tan\left( \dfrac{\lambda}{\sqrt{2}}\right) =\tanh\left( \dfrac{\lambda}{\sqrt{2}}\right) \,,
\end{equation*}
and $\lambda_4^2\approxeq4.73004$ is the least positive solution of
\begin{equation*}
\cos (\lambda ) \cosh (\lambda )=1.
\end{equation*}
Hence, we can conclude that $u^{(4)}(t)+M\,u(t)=0$ is  disconjugate  in $[0,1]$ if, and only if $M\in \left( -\left(  \lambda_4^2\right)  ^4,\left(  \lambda_4^1\right)  ^4\right) $.

 We point out that our result it is also applicable to other kind of operators, such as, for example $T_6[M]\,u(t)=u^{(6)}(t)-8\,u^{(3)}(t)+M\,u(t)$ on $[0,1]$. 
It is not difficult to verify, by means of the characterization of the first right point conjugate of $a$, given in Proposition \ref{P::4}, that $T_6[0]\,u(t)=0$  is a disconjugate equation on $[0,1]$. So, we can apply Theorem \ref{T::aut}. 

From the self-adjoint character of operator $T_6[0]$, one can conclude (see \cite{CabSaa} for details) that the eigenvalues related to boundary conditions $(2,4)$ and $(4,2)$, and $(5,1)$ and $(1,5)$ are the same. So, we only need to calculate the eigenvalues related to $(1,5)$, $(2,4)$ and $(3,3)$ boundary conditions. Numerically, we obtain that
\begin{itemize}
	\item the biggest negative eigenvalue related to the boundary conditions $(5,1)$ is $\lambda_1\approxeq -(8.40247)^6$.
	\item the least positive eigenvalue related to the boundary conditions $(4,2)$ is $\lambda_2\approxeq (6.717)^6$.
	\item the biggest negative eigenvalue related to the boundary conditions $(3,3)$ is $\lambda_3\approxeq -(6.2835)^6$.
\end{itemize}

Then we can conclude that $T_6[M]\,u(t)=0$ is a disconjugate equation on $[0,1]$ if, and only if, $M\in\left(-(6.717)^6,(6.2835)^6\right)$.

\vspace{0.3cm}

Let's consider now the operator
$T_4[M]\,u(t)=u^{(4)}(t)+50\,u''(t)+M\,u(t)$. In this case, if we study the operator for $M=0$, we obtain 
\[W_2^4[0](t)=\frac{5 \sqrt{2} t \sin \left(5 \sqrt{2} t\right)+2 \cos \left(5 \sqrt{2} t\right)-2}{2500}\,,\]
which
 changes sign on $[0,1]$. So, $T_4[0]\,u(t)=0$ is not a disconjugate equation on $[0,1]$.

But, if we take $\bar{M}=200$, we can verify, studying its different Wronskians, see Propositions \ref{P::3} and \ref{P::4}, that $T_n[200]\,u(t)=0$ is a disconjugate equation on $[0,1]$. Hence we can apply  Theorem \ref{T::aut} to this problem.

Due to the fact that it is also a self adjoint problem, we only need to obtain the eigenvalues related to the boundary conditions $(3,1)$ and $(2,2)$.

The eigenvalue related to the boundary conditions $(3,1)$ is given by $-\lambda_1^4$ where $\lambda_1\approxeq 3.71137$ is the least positive solution of the following equation
 \[\sqrt{25-\sqrt{425-\lambda ^4}} \sin \left(\sqrt{\sqrt{425-\lambda ^4}+25}\right)-\sqrt{\sqrt{425-\lambda ^4}+25} \sin \left(\sqrt{25-\sqrt{425-\lambda ^4}}\right)=0\,,\]
and the eigenvalue related to the boundary conditions $(2,2)$ is given by $\lambda_2^4$ where $\lambda_2\approxeq 2.77939$ is the least positive solution of the following equation
 {\scriptsize \[-2 \sqrt{200-\lambda ^4}+50 \sin \left(\sqrt{25-\sqrt{\lambda ^4+425}}\right) \sin \left(\sqrt{\sqrt{\lambda ^4+425}+25}\right)+2 \sqrt{200-\lambda ^4} \cos
	\left(\sqrt{25-\sqrt{\lambda ^4+425}}\right) \cos \left(\sqrt{\sqrt{\lambda ^4+425}+25}\right)=0\,.\]}

Hence we  conclude that  $T_4[M]\,u(t)=0$ is a disconjugate equation if, and only if, \[M\in (200-\lambda_2^4,200+\lambda_1^4)\approxeq(140.324,389.73)\,.\]

This characterization is also applicable to problems with non-constant coefficients. For instance, let's consider the  third order operator $T_3[M]\,u(t)=u^{(3)}(t)+\cos(10 t)\,u''(t)+M\,u(t)$ on $[0,1]$. 

Let's see that it is  disconjugate for $M=0$.

Since every solution of the first order linear differential equation $L_1\,u(t)=u'(t)+\cos(10\,t)\,u(t)=0$ follows the expression $u(t)= c_1\,e^{\sin(10\,t)}$, with $c_1\neq 0$, we conclude that it is disconjugate on any real interval. Also, it is well-known that the equation $L_2\,u(t)=u''(t)=0$ is also disconjugate on any real interval. So, as a direct application of Theorem \ref{T::1}, we can affirm that $L_2\,L_1\,u(t)=T_3[0]\,u(t)=0$ is a disconjugate equation on any real interval.

Now, using Propositions \ref{P::3} and \ref{P::4} again, we can obtain numerically the closest to zero eigenvalues related to the boundary conditions $(2,1)$ and $(1,2)$, which are $\lambda_1\approxeq -4.33149^3$ and $\lambda_2\approxeq 4.29055^3$, respectively. So, we can affirm that $T_3[M]\,u(t)=0$ is a disconjugate equation in $[0,1]$ if, and only if $M\in (-\lambda_2,-\lambda_1)$.

\end{document}